\documentclass[9pt]{amsart}

\usepackage{graphicx}
\usepackage{amsfonts}
\usepackage{amsmath}
\usepackage{amssymb}
\usepackage{mathrsfs}
\usepackage{enumerate}
\usepackage[hidelinks]{hyperref}

\DeclareMathOperator*{\support}{supp}

\DeclareMathOperator*{\psh}{PSH}

\DeclareMathOperator*{\vol}{Vol}
\DeclareMathOperator{\ddc}{dd^c}
\DeclareMathOperator*{\di}{d}
\DeclareMathOperator*{\dic}{d^c}
\DeclareMathOperator{\poly}{Poly}
\DeclareMathOperator{\spann}{span}
\DeclareMathOperator{\extr}{Extr}
\DeclareMathOperator{\Card}{card}

\renewcommand{\wp}{\mathscr{P}}

\newcommand{\ddcn}[1]{{(\ddc #1)}^{n}}

\newcommand{\areg}{A_{\text{reg}}}

\newcommand{\C}{\mathbb{C}}
\newcommand{\R}{\mathbb{R}}
\newcommand{\N}{\mathbb{N}}
\newcommand{\z}{\zeta}

\newtheorem{theorem}{Theorem}[section]

\newtheorem{proposition}{Proposition}[section]
\newtheorem{lemma}{Lemma}[section]

\theoremstyle{definition}
\newtheorem{definition}[theorem]{Definition}

\theoremstyle{remark}
\newtheorem{remark}[theorem]{Remark}

\title{The Extremal Plurisubharmonic Function of the Torus}

\author{Federico Piazzon}

\begin{document}
\maketitle
\begin{center}
\emph{In honor of the 60th birthday of my PhD advisor and friend,\\ Norm Levenberg.}
\end{center}

\begin{abstract}
We compute the extremal plurisubharmonic function of the real torus viewed as a compact subset of its natural algebraic complexification.
\end{abstract}

\section{Introduction}
Pluripotential Theory is the study of the complex Monge Ampere operator $\ddcn{}$ and plurisubharmonic functions. It is a non linear potential theory on multi-dimensional complex spaces that can be understood as the natural generalization of classical Potential Theory. Indeed, when the dimension $n$ of the ambient space is $1,$ the two theories coincide. We refer to the classical monograph \cite{Kl91} and to \cite{Le12} for an extensive review on the subject, including the more recent topics and developments. 

Pluripotential Theory is deeply related to complex analysis, approximation theory, algebraic and differential geometry, random polynomials and matrices, as discussed e.g.,  \cite{Le06,De07,BlLe13,Ba16,BlLeWi15}. Despite this rather wide range of applications, there are very few examples in which analytic formulas for the main quantities of Pluripotential Theory are available. The aim of the present work is to compute a formula for one of these objects, the \emph{extremal plurisubharmonic function} $V_{\mathbb T}(\cdot,\mathcal T)$ (see Definition \ref{extremaldef} below) of the set of real points of the two dimensional torus, i.e., 
$$\mathbb T:=\mathcal T\cap \R^3,$$
considered as a compact subset of its algebraic complexification 
\begin{equation}
\label{torusdef}
\mathcal T:=\{z\in \C^3: (z_1^2+z_2^2+z_3^2-(R^2+r^2))^2=4R^2(r^2-z_3^2)\},\;\; 0<r<R<\infty.
\end{equation}
Geometrically, the set $\mathbb T$ is obtained as surface of revolution rotating the real circle $\{(z_1-R)^2+ z_2^2=r^2,z_3=0\}\cap \R^3$ with respect to the $\Re z_3$ axis. Note that $\mathcal T$ can be equivalently described by
\begin{equation}
\label{torusdef2}
\mathcal T=\{z\in \C^3: (z_1^2+z_2^2+z_3^2+R^2-r^2)^2=4R^2(z_1^2+z_2^2)\},\;\; 0<r<R<\infty.
\end{equation}

The motivation of this investigation is three-fold. First, building examples of explicit formulas for extremal plurisubharmonic functions has its own interest since the known cases are so few. Second, the torus is a very classical set in various branches of Mathematics, so it would be interesting to specialize to the case of the torus certain applications of Pluripotential Theory, e.g., approximation of functions, orthogonal expansions, polynomial sampling inequalities and optimization, random polynomials, random arrays, and determinantal point processes. Lastly we mention our original motivation. Let $(M,g)$ be a compact real analytic Riemannian manifold. It has been proven that (an open bounded subset of) the Riemann foliation of the real tangent space $TM$ admits a natural complex structure such that the leaves of the Riemannian foliation glue together to construct a complex (Stein) manifold $X$ in which $M$ embeds as a totally real submanifold \cite{LeSz91}. The Stein manifold $X$ is termed the \emph{Grauert tube} of $M$. The most relevant two examples of unbounded Grauert tubes (see \cite{PaWo91} and references therein), are relative to the real sphere and real projective space. Their construction is in some sense canonical, not only from the Riemannian perspective, but also from the pluripotential point of view. In this last situation two elements are crucial: the algebraicity of the starting real manifold and the fact that the Baran metric \cite{BoLeWa08} ( a specific Finsler metric that can be defined starting from the extremal plurisubharmonic function) is Riemannian. 

Analogous investigations for the torus, that will be the subject of our future studies, necessitate the computation of the extremal plurisubharmonic function for the torus. This is accomplished in the present paper as stated in the following theorem.
\begin{theorem}\label{mainresult}
The extremal plurisubharmonic function of the torus (of radii $0<r<R<\infty$) is
\begin{align}\label{mainformula}
V_{\mathbb T}(z,\mathcal T)&:=\log h\Bigg\{\max\Bigg[ \frac{|\sqrt{1-(z_3/r)^2}+1|+|\sqrt{1-(z_3/r)^2}-1|}{2},\\
&\;\;\;\;\;\;\;\;\;\left|\frac{\sqrt{z_1^2+z_2^2}+z_2}{2(r+R)}\right|+\left|\frac{\sqrt{z_1^2+z_2^2}-z_2}{2(r+R)}\right|+\left|\frac{\sqrt{z_1^2+z_2^2}}{(r+R)}-1\right|\Bigg] \Bigg\},  \notag 
\end{align} 
where equality holds only for $z\in \mathcal T$ and $h(z):=z+\sqrt{z^2-1}$ is the inverse \emph{Joukowski function}, see \eqref{Joukowski} below.
\end{theorem}

\begin{remark}
It is worth saying that, even though the maximization procedure of formula \eqref{mainformula} may cause in principle a rather irregular behavior of $V_{\mathbb T}^*(\cdot,\mathcal T)$, the plot we made seem to reveal that indeed $V_{\mathbb T}^*(\cdot,\mathcal T)$ is quite smooth away from $\mathbb T.$ Since so far we have only partial results on this aspect, we display some sections of the graph of $V_{\mathbb T}^*(\cdot,\mathcal T)$ in Figure \ref{figureV}. Note in particular that the apparent jump in the derivative in the last two pictures of Figure \ref{figureV} is confined to the singular set of $\mathcal T$ (two leaves comes together) and it is actually only due to the choice of the branch of the square root when defining the local coordinates $(z_1,z_2)\mapsto(z_1,z_2,z^{(h)}(z_1,z_2))$ for $\mathcal T$ that are used in the proof of formula \eqref{mainformula}. This is evident in Figure \ref{Figuresmoothness} below, where the two branches are plotted together.
\end{remark}
 
\vskip 1cm
\textbf{Acknowledgements.} Obviously the present work has been deeply influenced by the discussions that I had with Norm Levenberg, both during and after my PhD period. This is fairly evident also by the number of papers of Professor Levenberg appearing in the references. Norm was the first one teaching me Pluripotential Theory and he is responsible of my fascination for the subject, which is a direct consequence of his enthusiastic lectures.

The first version of the present work contained an error which has been pointed out by Miros\l{}aw Baran during the \emph{multivariate polynomial approximation and pluripotential theory} section of the DRWA18 workshop. Sione Ma`u find out in few hours how to fix the issue: the last part of the proof of Theorem \ref{mainresult} is essentially due to him.  
\vskip 1cm
\subsection{Pluripotential Theory in the Euclidean setting} 
Before proving equation \eqref{mainformula}, we recall for the reader's convenience some notation, definitions, and basic facts from Pluripotential Theory in different settings. 

\emph{Plurisubharmonic functions} on a domain $\Omega\subset\C^n$ are functions that are globally uppersemicontinuous and subharmonic on the intersection with $\Omega$ of any complex line (or analytic disk in $\Omega$). We denote this set of functions by $\psh(\Omega).$ If $u\in\mathscr C^2(\Omega)$ the condition $u\in \psh(\Omega)$ reduces to the positivity of the $(1,1)$ differential form
$$\ddc u=2 i\partial\bar \partial u:= \sum_{j=1}^n \frac{\partial^2 u}{\partial z_j\partial\bar  z_j}(z) dz_j\wedge d\bar z_j.$$    
Here we used the classical notation for exterior and partial differentiation in $\C^n$, i.e.,
\begin{align*}
 &\di:=(\partial+\bar\partial)\;,\;\;\;\;\;\;\dic:=i(\bar\partial-\partial ),\\
 &\partial u:=\sum_{j=1}^n \frac{\partial u}{\partial z_j}dz_j\;,\;\;\;\;\;\;\bar \partial u:=\sum_{j=1}^n \frac{\partial u}{\partial \bar z_j}d\bar z_j
 \end{align*}
If a function $u$ is uppersemicontinuous and locally integrable, we can still check if it is a plurisubharmonic functions by checking the positivity of the \emph{current} (i.e., differential form with distributional coefficients) $\ddc u.$ We refer the reader to \cite{Kl91} for an extensive treatment of the subject.

The complex \emph{Monge Ampere} operator $\ddcn{}$ can be defined for $\mathscr C^2(\Omega)$ functions by setting
$$\ddcn{u}:=\ddc u\wedge \ddc u \wedge \dots\dots \wedge \ddc u= c_n \det [\partial_h\bar\partial_k u]_{h,k} \di{\vol}_{\C^n}.$$ 
Note that, in contrast with $\ddc,$ this is a fully non linear differential operator, therefore its extension to functions that are not $\mathscr C^2(\Omega)$ is highly non-trivial..
 
In the seminal paper \cite{BeTa82}, Bedford and Taylor  extended such a definition to locally bounded plurisubharmonic functions by an inductive procedure on the dimension $n.$ In this more general setting $\ddcn{u}$ is a \emph{measure}. Further extensions have been carried out more recently \cite{Ce98}.

Among all plurisubharmonic functions, \emph{maximal plurisubharmonic functions} play a very special role: one can think of the relation of this subclass with $\psh$ as an analogue to the relation of harmonic functions with subharmonic functions in $\C.$ Indeed a plurisubharmonic function $u$ is maximal in $\Omega$ if for any subdomain $\Omega'\subset \Omega$ and for any $v\in \psh(\Omega)$ such that $u(z)\geq v(z)$ for any $z\in \partial \Omega'$, it follows that $u\geq v$ in $\Omega'.$  Perhaps more importantly to our aims, a maximal plurisubharmonic function $u$ on $\Omega$ is characterized by (being plurisubharmonic and) $\ddcn{u}\equiv 0$ on $\Omega$ in the sense of measures.

In Pluripotential Theory there is an analogue of the Green Function with pole at infinity. Let $K\subseteq \C^n$ be a compact set. One can consider the upper envelope
\begin{equation}\label{zaharjuta}
V_K(z):=\sup\{u(z): u\in \mathcal L(\C^n), u(w)\leq 0\; \forall w\in K\},
\end{equation}
where $\mathcal L(\C^n)$ is the Lelong class of plurisubharmonic functions on $\C^n$ having logarithmic growth, i.e., $u\in \mathcal L(\C^n)$ if $u$ is plurisubharmonic on $\C^n$ and for any $M$ large enough there exists a constant $C$ such that $u(z)\leq \log|z|+C$ for any $|z|>M.$
There are two possible scenarios. Either $V_K$ is not locally bounded and in this case the set $K$ is termed pluripolar (roughly speaking it is too small for pluripotential theory). Or $V_K$ is locally bounded and its uppersemicontiunuous regularization
\begin{equation}\label{zaharjuta*}
V_K^*(z):=\limsup_{w\to z}V_K(w),
\end{equation}  
which is called \emph{Zaharjuta extremal plurisubharmonic function}, is a locally bounded plurisubharmonic function on $\C^n$ which is maximal on $\C^n\setminus K$, that is
$$\ddcn{V_K^*}=0 \text{ on }\C^n\setminus K.$$ 

One important fact about $V_K$ is that one can recover the same function by a different upper envelope. More in detail, following \cite{Si81}, we have 
$$V_K(z)=\log^+\Phi_K(z),$$
where the \emph{Siciak extremal function}  $\Phi_K$ is defined by setting
\begin{equation}\label{Siciak}
\Phi_K(z):=\lim_{j\to \infty}\left(\sup\{|p(z)|, p\in\wp^j(\C), \|p\|_K\leq 1\}\right)^{1/j}.
\end{equation}
Here $\wp^j(\C^n)$ denotes the space of polynomial functions with complex coefficients on $\C^n$ whose total degree is at most $j\in \N$ and  $\|\cdot\|_K$ denotes the uniform norm on $K.$

The function $V_K^*(z)=\log\Phi_K^*(z)$ is referred as the \emph{pluricomplex Green function with pole at infinity} of $\C^n\setminus K$ or as the \emph{extremal plurisubharmonic function} of $K.$ 

Computing extremal plurisubharmonic functions is, in general, a very hard task. There are very few examples (see \cite{BoLeMaPi17,BlBoLeMaPi18}) of settings in which it has been computed by various (hard to be generalized) techniques. Numerical methods for the approximation of extremal functions have been developed in \cite{Pi18}. An exceptional case is the one of $K$ being a centrally symmetric \emph{convex real body}. In such a case the Baran formula (see \eqref{Baran} below) gives an analytic expression for $V_K^*$, moreover even deeper properties of this function have been  
studied (see \cite{BuLeMa05,BuLeMaRe10,BuLeMa15}) as the structure of the Monge Ampere foliation of $\C^n\setminus K$ and the density of $\ddcn{V_K^*}$ with respect to the $n$ dimensional Lebesgue measure on $K$, \cite{Ba92,Ba92b}. 

Let $K\subset \R^n$ be a real convex body. We denote by $K^*$ its polar set
$$K^*:=\{x\in \R^n: \langle x;y\rangle\leq 1,\forall y\in K\}$$
and by $\extr K^*$ the set of its \emph{extremal points}. That is, $x\in \extr K^*$ if $x\in K^*$ and $x$ is not an interior point of any segment lying in $K^*$ (note that $K^*$ is convex by definition but it may be not strictly convex).

We denote by $h:\C\rightarrow \C$ the \emph{inverse Joukowski map}
\begin{equation}\label{Joukowski}
h(z):z+\sqrt{z^2-1}.
\end{equation}
Baran \cite{Ba92} proved that, for any centrally symmetric convex real body,
\begin{equation}\label{Baran}
V_K^*(z)=\sup_{y\in \overline{\extr K^*}}\log |h(\langle z,y\rangle)|.
\end{equation} 

\subsection{Pluripotential Theory on algebraic varieties}
Let $A$ be an algebraic variety having pure dimension $m,$ $1\leq m<n.$  Take any set of local defining function $f_1,f_2,\dots,f_k$ for $A$ and consider the function
$$f(z):=\max_{j\in \{1,2,\dots,k\}}\log |f_j(z)|.$$
This is a a plurisubharmonic function  on $\C^n$ such that $A$ is (locally) contained in the $\{-\infty\}$ set of it. This is (see for instance \cite{Kl91}) equivalent to the fact that $A$ is locally pluripolar, that is \emph{too small for $n$ dimensional pluripotential theory}. On the other hand, the set $\areg$ of the regular points of $A$ is an $m-$dimensional complex manifold so $m$-dimensional pluripotential theory is well defined on it by naturally extending the definitions given in the Euclidean setting using local holomorphic coordinates. Then, using the nice properties of coordinate projections \cite{Ch89}, it is possible to define pluripotential theory (and in particular the complex Monge Ampere operator and an associated capacity) on the set $A,$ \cite{Ze91,De85,Be82}.

Up to this point, this procedure can be carried out on analytic sets. In contrast, if we want to deal with the restriction of polynomials of $\C^n$ to $A$, the algebraicity of $A$ becomes determining as shown by the following fundamental result.
\begin{theorem}[Sadullaev \cite{Sa82}]\label{SadResult}
Let $A\subset\C^n$ be a irreducible analytic subvariety of $\C^n$ of dimension $1\leq m<n.$ Then $A$ is algebraic if and only if there exists a compact set $K\subset A$ such that
\begin{equation}
V_K(z,A):=\log^+\lim_{j\to \infty}\left(\sup\left\{|p|,p\in \wp^j(\C^n):\;\|p\|_K\leq 1\right\}\right)^{1/j}
\end{equation} 
is locally bounded on $A.$
If this is the case, then $V_L(\cdot,A)$ is locally bounded for any compact set $L\subset A$ such that $L\cap \areg$ is not pluripolar in $\areg$ and the function
\begin{equation}\label{Sadullaev}
V_L^*(z,A):=\limsup_{\areg\ni \z\to z}V_L(\z,A)
\end{equation}
is maximal in $A\setminus L,$ i.e.,
\begin{equation}
(\ddc V_L^*(z,A))^m=0,\;\;\text{ in }A\setminus L.
\end{equation}
\end{theorem} 
The extremal plurisubharmonic function $V_K^*(\cdot,A)$ defined by Sadullaev can be understood as the natural counterpart of the (log of the) Siciak type extremal function defined in \eqref{Siciak}. The definition of a Zaharjuta type (i.e., built by an upper envelope of plurisubharmonic functions of a given growth) extremal function in this setting can be view as a particular case of the \emph{pluricomplex green function for Stein Spaces with parabolic potential}, developed in \cite{Ze91}.
Let $A\subset\C^n$ an algebraic irreducible variety of pure dimension $m<n$, then we denote by $\mathcal L(A)$ the class of locally bounded functions $u$ on $A$ that are  plurisubharmonic on $\areg$ such that for some constant (depending on $u$) we have
$$u(z)<\max_{j=1,2,\dots, n}\log|z_i|+c_u$$
for any $z\in A$ and $|z|$ large enough. In this \emph{simplyfied} setting we have the following.
\begin{theorem}[Zeriahi \cite{Ze91}]
Under the above assumption the Siciak and the Zaharjuta type extremal functions are the same, that is
\begin{equation}\label{Zeriahi}
V_K^*(z,A)=\limsup_{\z\to z}\sup\left\{u(\z)\in \mathcal L(A): u\leq 0\text{ on } K  \right\}
\end{equation}
for any compact set $K\subset A$.
\end{theorem}
\begin{definition}[Plurisubharmonic extremal function]\label{extremaldef}
Since in our setting of an irreducible algebraic variety of pure dimension $m<n$ embedded in $\C^n$ the extremal functions \eqref{Sadullaev} and \eqref{Zeriahi} are the same, we refer to both of them as \emph{the plurisubharmonic extremal function} of $K.$
\end{definition}
\begin{remark}
We \textbf{warn the reader} that, for the sake of an easier presentation of the definitions and results, we restrict our attention to the case of \emph{irreducible} algebraic varieties. This avoids some ambiguity in the considered class of plurisubharmonic functions (namely any weakly plurisubharmonic function is actually plurisubharmonic \cite{De85}) and consequently in the definition of the extremal functions. Indeed, in our setting $u\in \psh(A)$ equivalently means that
\begin{itemize}
\item $u$ is the restriction to $A$ of a plurisubharmonic function $\tilde u:\Omega\rightarrow [-\infty,\infty[$, for some neighbourhood $\Omega$ of $A$ in $\C^n$ or 
\item the restriction $\hat u$ of the uppersemicontinuous function $u$ to the regular points $\areg$ of $A$ is a plurisubharmonic function on the complex manifold $\areg$.
\end{itemize}
\end{remark}
\subsection{Polynomial degrees and P-pluripotential Theory}
When dealing with multivariate polynomials, the concept of degree of a polynomial needs to be specified. Although the standard choice is to use the so-called \emph{total degree} $\deg$ (i.e., $\deg(z^\alpha):=|\alpha|_1$ and $\deg(\sum_{i=1}^{k} c_i z^{\alpha})=\max_{1\leq i\leq k}\{\deg(z^{\alpha_i}):c_i\neq 0\}$), for various applications it may be convenient to use different definitions of "degree". Classical examples of this are the so called \emph{tensor degree} and \emph{Euclidean degree} \cite{Tr17}, corresponding to the use of $\ell^\infty$ and $\ell^2$ norms instead of the $\ell^1$ norm in the definition of the function $\deg$, respectively. 

It  is clear that the definition of the Siciak type extremal function \eqref{Siciak} depends on which definition of degree is used, but the Lelong class $\mathcal L(\C^n)$ used in defining the Zaharjuta type extremal function \eqref{zaharjuta} a priori does not.

More in general one can define a degree $\deg_P$ on $\wp(\C^n)$ depending on a subset $P$ of $\R_+^n:=\{x\in \R^n:x_i\geq 0\;\forall i=1,2,\dots,n\}$  satisfying certain geometric properties. The study of this variant of Pluripotential Theory has been started very recently in \cite{BoLe17p,Ba17} and it is termed \emph{P-pluripotential Theory}. We recall a few facts about P-pluripotential Theory that we will need to use later on.

Let $P\subset \R_+^n$ be a convex set containing a neighborhood of $0$  (in the relative topology of $\R_+^n$). We denote consider the following polynomial complex vector spaces
\begin{equation}
\poly(k P):=\spann\{z^\alpha: \alpha\in k P\cap \N^n\},\;\forall k\in \N.
\end{equation}
Equivalently one can set
\begin{equation}
\deg_P(z^\alpha):=\inf_{k\in \N}\{k:\alpha\in kP\},\;\;\deg_P(\sum_{\alpha\in I}c_\alpha z^\alpha):=\max_{\alpha\in I,c_\alpha\neq 0}\deg_P(z^\alpha),
\end{equation} 
and $p\in \poly(kP)$ if and only if $p$ is a polynomial and $\deg_P(p)\leq k.$
Let us recall that the \emph{support function} $\phi_P$ of the convex set $P$ is defined as
$$\phi_P(x):=\sup_{y\in P}\langle x;y\rangle,\;\forall x\in \R_+^n.$$
We can use the support function of $P$ to introduce a dependence on $P$ in (a modified version of) the definition of the Lelong class. Namely, we define the \emph{logarithmic support function} $H_P$ and the $P$-Lelong class
\begin{align}
&H_P(z):=\phi_P(\log|z_1|,\log|z_2|,\dots,\log|z_n|).\label{logaritmicsupport}\\
&\mathcal L_P(\C^n):=\{u\in \psh(\C^n):\; u-H_P\text{ is bounded above for }|z|\to \infty\}.
\end{align}
Note that the standard case of total degree polynomials corresponds to picking $P=\Sigma$, the standard unit simplex. Indeed we have $\Sigma=\{y\in \R_+^n:|y|_1\leq 1\}$  and $\phi_\Sigma(x)=|x|_\infty$, so that $H_P(z)=\max_i\log|z_i|$ and $\mathcal L_\Sigma(\C^n)$ reduces to the classical Lelong class.

We can introduce a Zaharjuta type extremal function for any such $P$ and any compact set $K\subset \C^n$ setting
\begin{align}\label{P-zaharjuta}
&V_{K,P}(z):=\sup\{u(z):u\in \mathcal L_P(\C^n), u|_K\leq 0\},\\
&V_{K,P}^*(z):=\lim_{\z\to z}V_{K,P}(\z),
\end{align}
and a Siciak type extremal function by setting
\begin{equation}
\Phi_{K,P}(z):=\lim_k\sup\{|p|^{1/k}: p\in \poly(kP),\|p\|_K\},
\end{equation}
where the existence of the limit is part of the statement (see \cite{Ba17}) and if the limit is continuous the convergence improves from point-wise to locally uniform.

For any non pluripolar set $K$ we have
\begin{equation}
V_{K,P}(z)=\log^+ \Phi_{K,P}(z).
\end{equation}
Many results of Pluripotential Theory have been extended to the P-pluripotential setting. Among others, we recall for future use this version of the Global Domination Principle.
\begin{proposition}[P-Global Domination Principle \cite{LePe18}]\label{Pdomination}
Let $u\in \mathcal L_P(\C^n)$ and $v\in \mathcal L_P^+(\C^n),$ i.e. $v\in \mathcal L_P(\C^n)$ and $|v-H_P|(z)$ is bounded for $|z|\to \infty.$ Assume that $u\leq v$ $\ddcn{v}$-a.e., then 
$$u(z)\leq v(z),\;\forall z\in \C^n.$$
\end{proposition} 
\begin{figure}
\caption{ Plots of $V_{\mathbb T}^*((z_1,z_2,z_3^{h}(z_1,z_2)),\mathcal T).$ Here $(z_1,z_2)\mapsto z_3^{h}(z_1,z_2),$ $h=1,2,3,4$ are the four leaves of $\mathcal T$, that is the four local inverses of the coordinate projection $\mathcal T\ni z\mapsto (z_1,z_2).$ By symmetry we need to look only at two of these leaves (i.e., $h=1,2$), as the graph of $V_{\mathbb T}^*$ is the same on the other two. From left to right and from above to below, we plot $\R^2\ni(z_1,z_2)\mapsto V_{\mathbb T}^*((z_1,z_2,z_3^{1}(z_1,z_2)),\mathcal T),$ $\R^2\ni(z_1,z_2)\mapsto V_{\mathbb T}^*((z_1,z_2,z_3^{2}(z_1,z_2)),\mathcal T),$ $\R^2\ni(\Re z_1,\Im z_1)\mapsto V_{\mathbb T}^*((z_1,0,z_3^{1}(z_1,0)),\mathcal T),$ $\R^2\ni(\Re z_1,\Im z_1)\mapsto V_{\mathbb T}^*((z_1,0,z_3^{2}(z_1,0)),\mathcal T).$}
\begin{center}
\label{figureV}
\begin{tabular}{cc}
\includegraphics[scale=0.4]{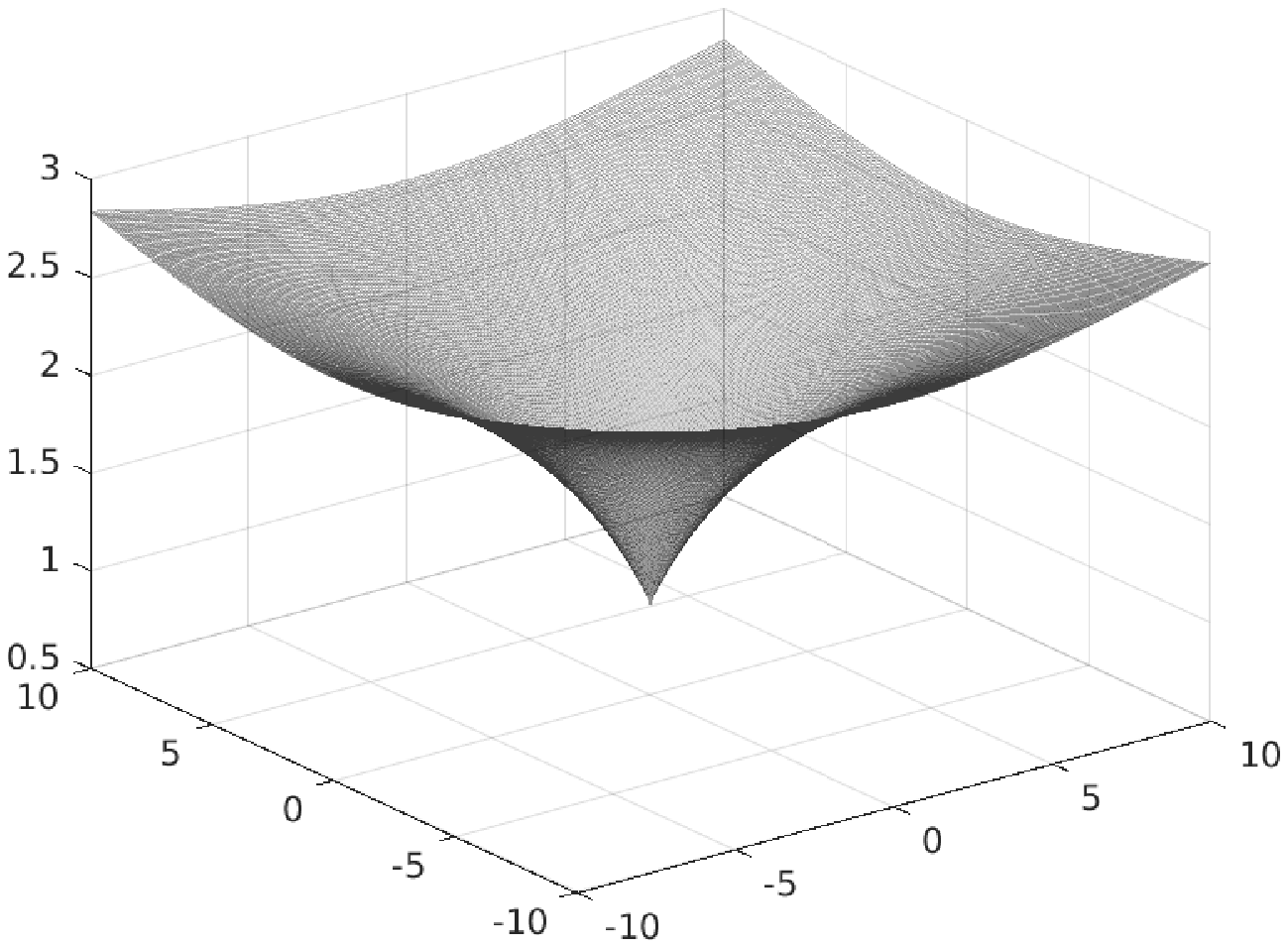}&\includegraphics[scale=0.4]{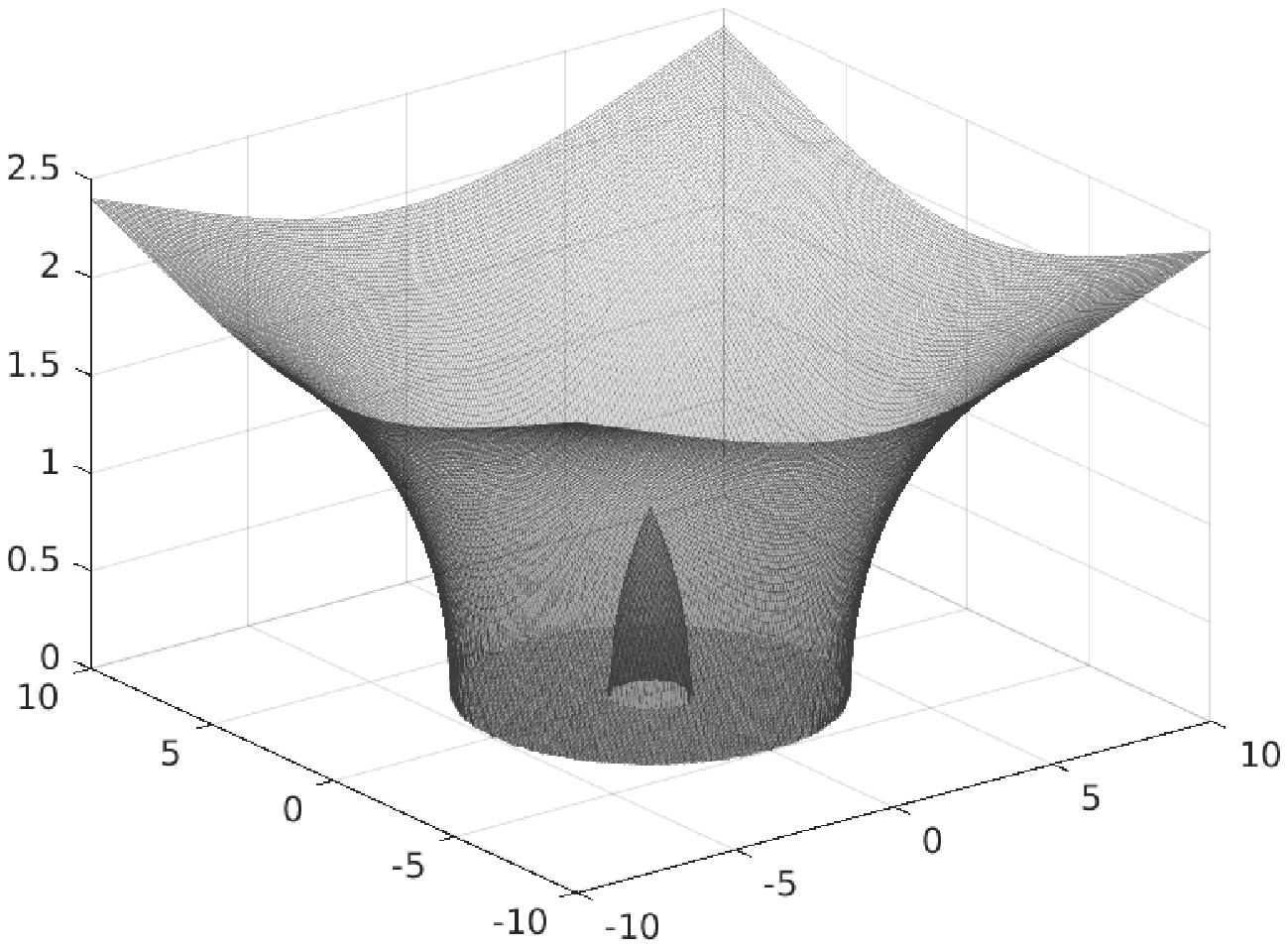}\\
\includegraphics[scale=0.4]{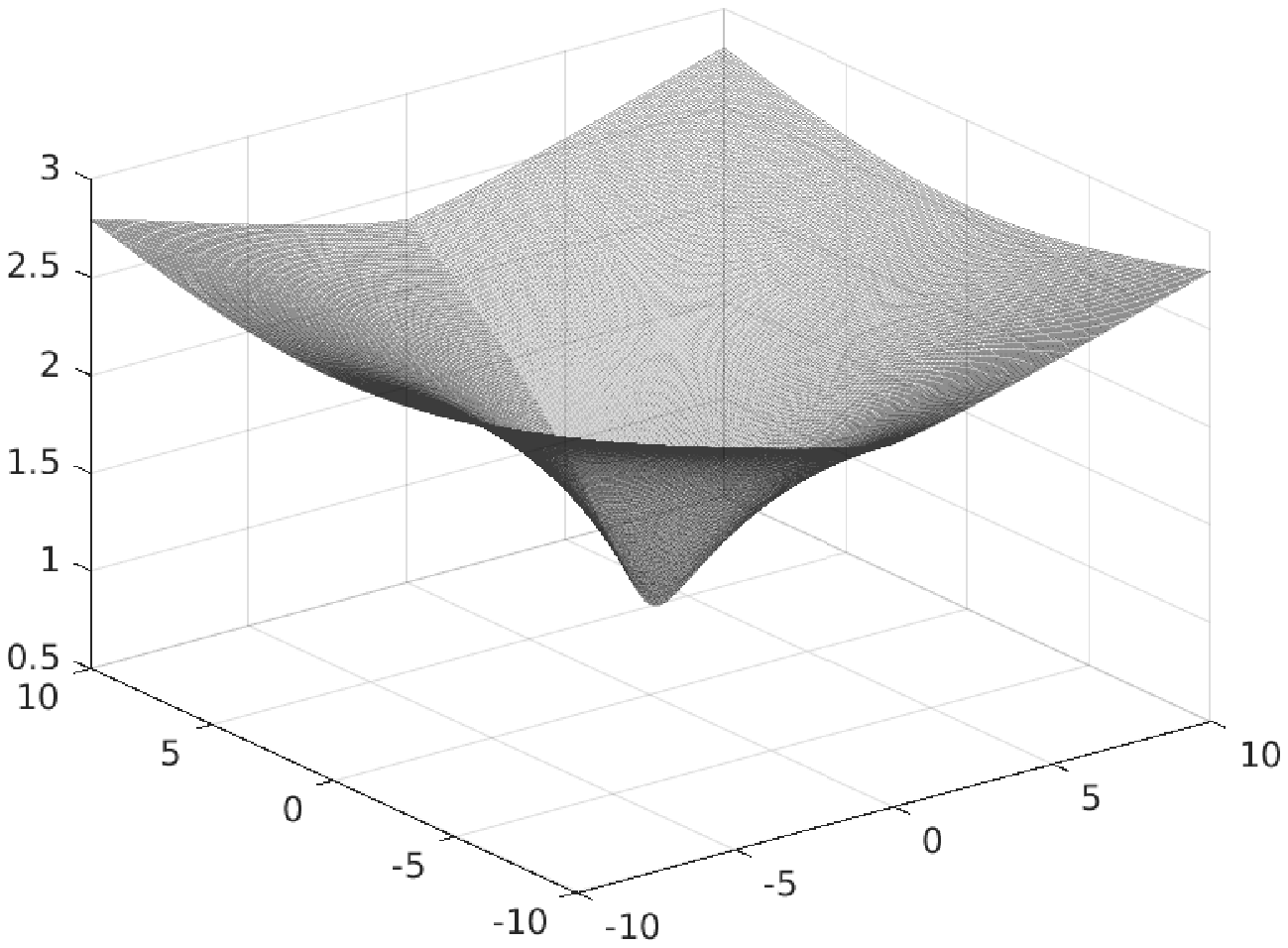}&\includegraphics[scale=0.4]{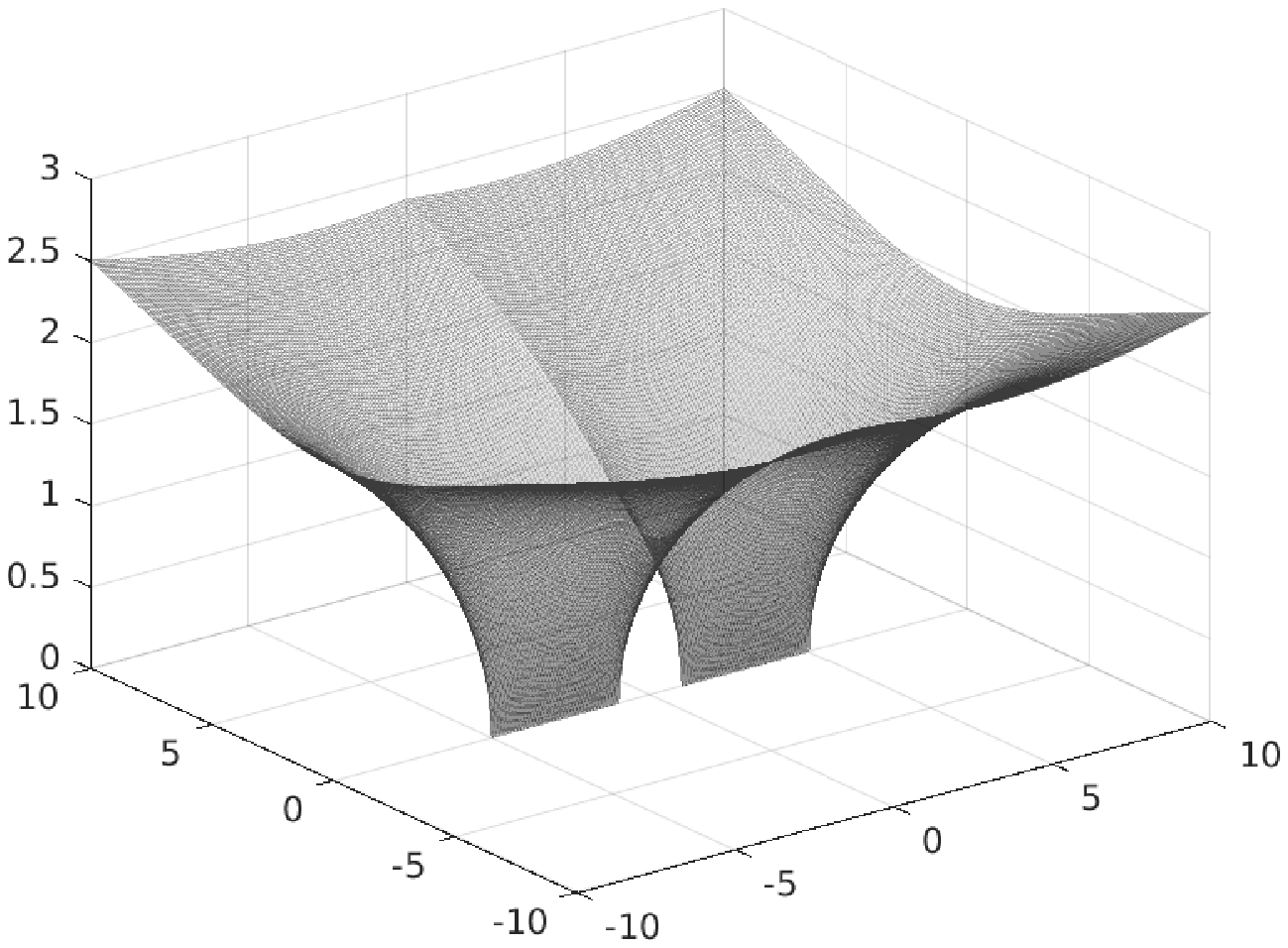}
\end{tabular}
\end{center}
\end{figure}

\begin{figure}
\caption{The multifunctions $\R^2\ni(\Re z_1,\Im z_1)\mapsto \{V_{\mathbb T}^*((z_1,0,z_3^{(1)}(z_1,0)),\mathcal T),V_{\mathbb T}^*((z_1,0,z_3^{(2)}(z_1,0)),\mathcal T)\}$ (above) and $\R^2\ni(\Re z_1,\Re z_2)\mapsto \{V_{\mathbb T}^*((z_1,z_2,z_3^{(1)}(z_1,z_2)),\mathcal T),V_{\mathbb T}^*((z_1,z_2,z_3^{(2)}(z_1,z_2)),\mathcal T)\}$ (below) exhibit a nice smoothness away from $\mathbb T.$ }
\begin{center}
\label{Figuresmoothness}
\includegraphics[scale=0.36]{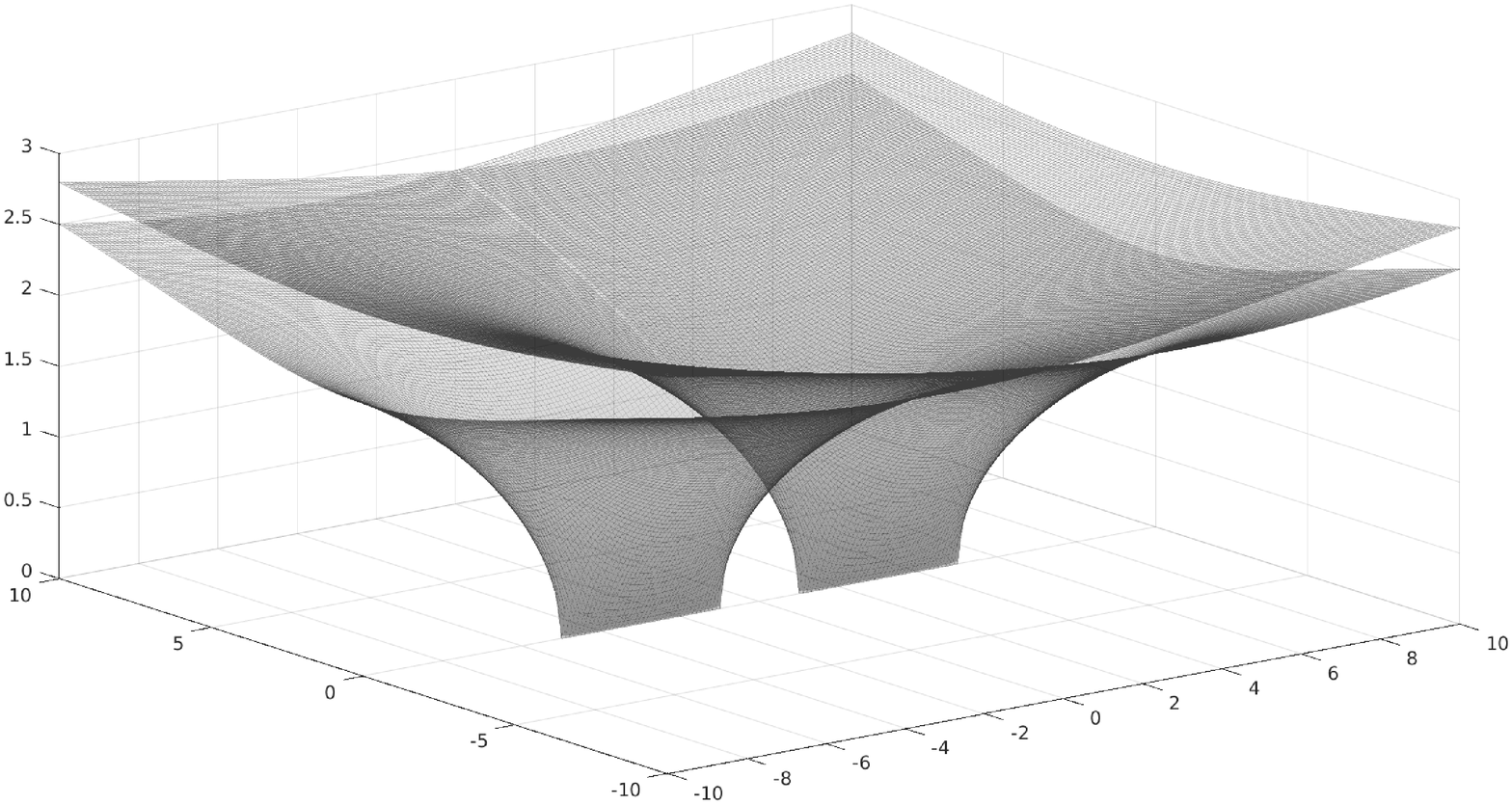}\\
\includegraphics[scale=0.36]{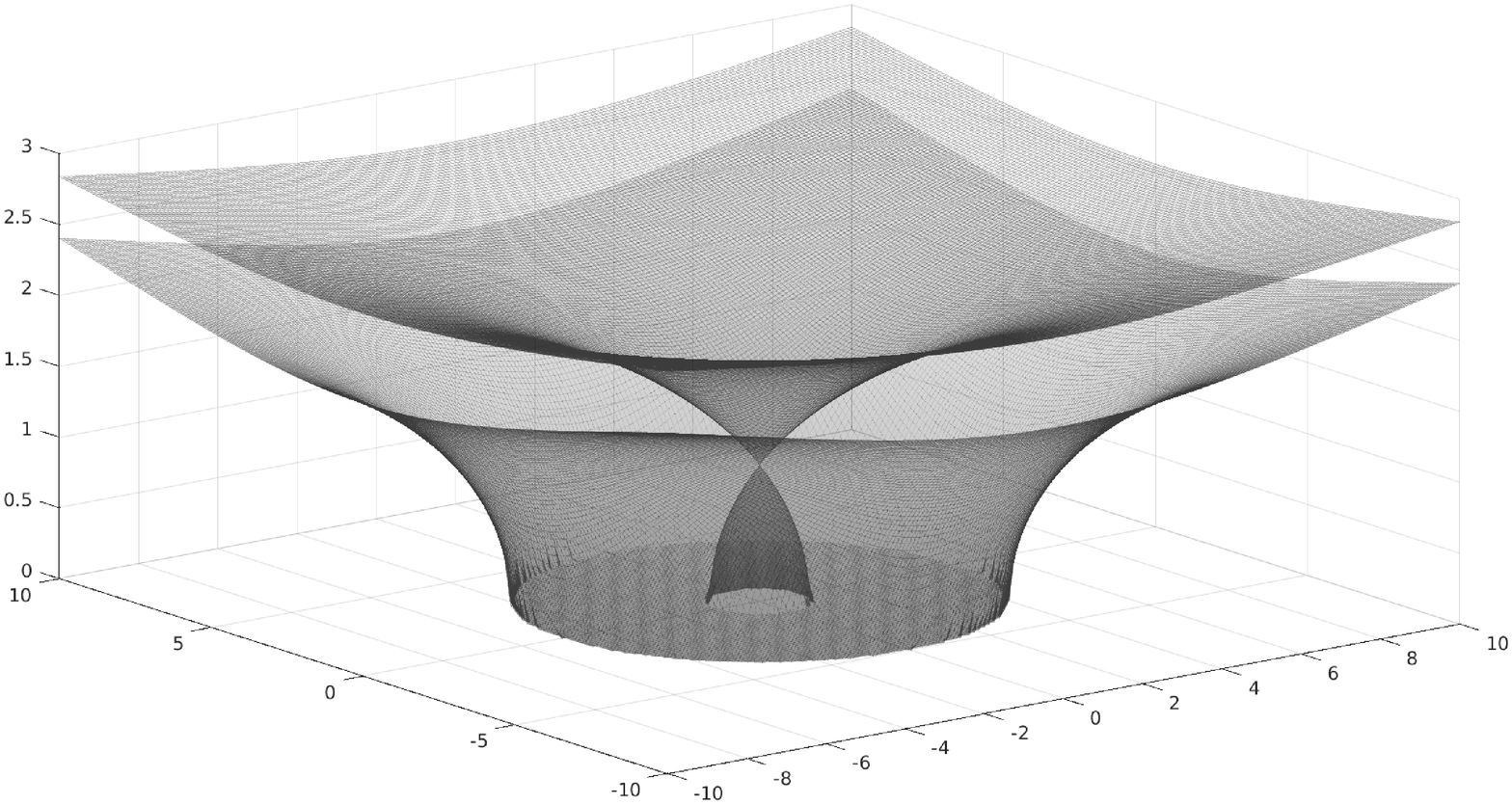}
\end{center}
\end{figure}

\section{Proof of Theorem \ref{mainresult}}
We divide the proof in some steps since they might have their own interest.
Let $\Psi:\C^3\rightarrow \C^3$ be defined by
$$\Psi(z):=\left(\begin{array}{c}\frac{z_1^2+z_2^2+z_3^2+R^2-r^2}{2R}\\z_1^2+z_2^2\\z_2^2  \end{array}\right)=:\left(\begin{array}{c}w_1\\w_2\\w_3  \end{array} \right).$$
This can be used as a polynomial change of coordinates, indeed we have $\Psi(\mathcal T)=\mathcal C,$ where $\mathcal C$ is the parabolic cylinder
$$\mathcal C:=\{w\in \C^3: w_1^2=w_2  \}.$$
Indeed, by elementary algebraic manipulation, we can derive by \eqref{torusdef} the equivalent representation of $\mathcal T$ as 
$$\mathcal T=\left\{z\in \C^3: \left( \frac{z_1^2+z_2^2+z_3^2+R^2-r^2}{2R}\right)^2=z_1^2+z_2^2\right\}. $$
Notice also that $\Psi$ has $8$ (possibly coinciding) inverses determined by 
$$\Psi^{\leftarrow}(w)=\left\{ \left(\begin{array}{c}\alpha_1\sqrt{w_2-w_3}\\\alpha_2\sqrt{w_3}\\ \alpha_3\sqrt{2R w_1-w_2+r^2-R^2}\end{array}\right): \alpha\in \{-1,1\}^3\right\}.$$  
\begin{lemma}
In the above notation we have
\begin{equation}\label{VTVP}
V_{\mathbb T}^*(z,\mathcal T)=\frac{1}{2}V_E^*(\Psi(z),\mathcal C),
\end{equation}
where
$$E:=\{w\in \mathcal C\cap \R^3: 0\leq w_3\leq w_1^2,\;R-r\leq w_1\leq R+r\}.$$
In particular, the function $V_{\mathbb T}^*(\cdot,\mathcal T)$ is constant on $\Psi^{-1}(w)$ for any $w\in \mathcal C.$ 
\end{lemma} 
\begin{proof}
The mapping $\Psi$ is a polynomial mapping having the same (total) degree in each component and whose homogeneous part
$$\hat\Psi(z):=\left(\begin{array}{c}\frac{z_1^2+z_2^2+z_3^2}{2R}\\z_1^2+z_2^2\\z_2^2  \end{array}\right)$$
clearly satisfies the \emph{Klimek condition}
$$\hat \Psi^{\leftarrow}(0)=0.$$
Moreover it is straightforward to check that $\Psi(\mathbb T)=E$ and $\Psi^{\leftarrow}(E)=\mathbb T.$

Therefore we can apply \cite[Th. 5.3.1]{Kl91} in its "equality case" to get \eqref{VTVP}. 
\end{proof}
\begin{remark}
Note that \cite[Th. 5.3.1]{Kl91} is not formulated in terms of varieties or complex manifolds, it holds for honest extremal functions in Euclidean space. Nevertheless we can apply it. Indeed the statement of the theorem does not require the compact set of which we are computing the extremal function to be non pluripolar and, due to the deep result of Sadullaev (see Prop. \ref{SadResult} and \cite{Sa82}), we can recover the extremal function of a compact subset of an algebraic variety by using \emph{global} polynomials.
\end{remark}

\begin{proposition}
Let $\pi:\mathcal C\rightarrow \C^2,$ $\pi(w):=(w_1,w_3).$ Then we have
\begin{equation}\label{VPVS}
V_E^*(w,\mathcal C)=V_{\pi(E),\Sigma_{2,1}}(\pi(w)),\;\;\forall w\in \mathcal C,
\end{equation}
where $\Sigma_{2,1}:=\{y\in \R_+^2: y_1/2+y_2\leq 1\}.$
\end{proposition}
\begin{proof}
By definition we have $V_E(w,\mathcal C)=\sup\{u(w),u\in \mathcal E\},$ where 
$$\mathcal E:=\{u\in \mathcal L(\mathcal C), u|_E\leq 0\},$$
while $V_{\pi(E),\Sigma_{2,1}}(\eta)=\sup\{v(\eta),u\in \mathcal F\},$ where
$$\mathcal F:=\{v\in \psh(\C^2),v-\max\{2\log|w_1|,\log|w_3|\}\text{ bounded above as }|w|\to \infty, v|_{\pi E}\leq 0\}.$$
Here we used that the logarithmic support function (cf. \eqref{logaritmicsupport}) for $\Sigma_{2,1}$ is $H_{\Sigma_{2,1}}(z):=\max\{2\log|\eta_1|,\log|\eta_2| \}.$

We use the holomorphic coordinates $\z:=(w_2,w_3)$ for $\mathcal C.$ These coordinates are a set of so called Rudin coordinates \cite{Ru68}, i.e., $\mathcal C\subset\{w\in \C^3: |w_1|\leq C (1+|\z|)\}$ for some constant $C.$ This choice allows us to re-define the Lelong class $\mathcal L(\mathcal C)$ as
\begin{align*}
\mathcal L(\mathcal C)=&\{u\in \psh(\mathcal C): u-\max\{\log|w_2|,\log|w_3|\}\text{ bounded above as }|w|\to \infty\}\\
=&\{u\in \psh(\mathcal C): u-\max\{2\log|w_1|,\log|w_3|\}\text{ bounded above as }|w|\to \infty\}.
\end{align*}

Also note that the coordinate projection $\pi:\mathcal C\rightarrow \C^2$, $\pi(w):=(w_1,w_3)$ is one to one. It follows that, if $u\in \mathcal E$, then $u\circ \pi^{-1}\in \mathcal F$ and, if $v\in \mathcal F$, then $v\circ \pi\in \mathcal E.$ Therefore the two upper envelopes coincide when composed with the coordinate projection map.
\end{proof}


\begin{proposition}\label{piEK}
Let 
$$K:=\{x\in \R^2:-1\leq x_1\leq 1,\;-(r x_1+R)\leq x_2\leq rx_1+R\},$$
then we have
\begin{equation}
V_{\pi E,\Sigma_{2,1}}(t)=2V_K(\Phi(t)), \forall t\in \C^2,
\end{equation}
where
$$\Phi\left(\begin{array}{c} t_1\\t_2  \end{array}\right)=\left(\begin{array}{c} \frac{t_1-R}r\\\sqrt{t_2}  \end{array}\right)$$
and, by symmetry, we can choose \emph{any} branch of the square root.
\end{proposition}
\begin{proof}
Let $v(t):=2 V_K^*(\Phi(t)).$ This is a plurisubharmonic function on $\C^2\setminus \{t\in \C^2:t_2\in \R, t_2<0\}$ because it is the composition of a plurisubharmonic function with an holomorphic map. Note that $N:=\{t\in \C^2:t_2\in \R, t_2\leq0\}$ is a pluripolar set in $\C^2.$

Due to the symmetry of $V_K^*(t_1,t_2)=V_K^*(t_1,-t_2)$ (this is a byproduct of the proof of Theorem \ref{mainresult} below), the uppersemicontinuous extension to $\C^2$ of $v|_{\C^2\setminus N}$ is a continuous plurisubharmonic function, that indeed coincides with $v.$ Hence $v$ is plurisubharmonic on $\C^2.$

It is a classical result (see for instance \cite{Kl91}) that $V_K^*\in \mathcal L^+(\C^2).$ More precisely, there exist $C\in \R$ such that, for $|w|$ large enough, we have
\begin{equation}
|V_K^*(w)-\max\{\log|w_1|,\log|w_2|\}|<C.
\end{equation}
It follows that
$$|2V_K(\Phi(t))-2\max\{\log|\Phi_1(t)|,\log|\Phi_2(t)|\}|\sim|2V_K(\Phi(t))-\max\{2\log|t_1|,\log|t_2|\}|<C,$$
since $v\in \psh(\C^2)$ we get $v\in \mathcal L_P^+(\C^2).$

Now note that $v=0$ on $\pi E.$ This follows by the fact that $\Phi(\pi E)\subseteq K$ and $V_K^*=0$ on $K.$ This last statement will clarified in the "proof of Theorem \ref{mainresult}" below, where the function $V_K^*$ is computed.  

Now we notice that $(\ddc v)^2$ is zero on $(\C^2\setminus N)\setminus \pi E.$ This follows by the fact that $\Phi$ is holomorphic and on $\C^2\setminus N$ and $v=2 V_K^*\circ \Phi$, where $V_K^*$ is maximal on $\C^2\setminus K\supset \C^2\setminus \Phi(\pi E).$ Since $N$ is pluripolar and $v$ is locally bounded we have $(\ddc v)^2=0$ on $\C^2\setminus \pi E$, \cite{BeTa82}. Hence $\support (\ddc v)^2\subseteq \pi E.$

We can apply Proposition \ref{Pdomination} with $u(t):=V_{\pi E,\Sigma_{2,1}}^*(t)$ which is by definition a function in $ \mathcal L_P(\C^2)$ q.e. vanishing on $\pi E.$ Indeed, since the Monge Ampere of locally bounded $\psh$ function does not charge pluripolar sets \cite{BeTa82}, we have $u\leq v$ almost everywhere with respect to $(\ddc v)^2$.  We conclude that
$$V_{\pi E,\Sigma_{2,1}}^*(t)\leq 2V_K^*(\Phi(t)),\;\;\forall t\in \C^2.$$
But, since the left end side is (q.e.) defined by an upper envelope containing the right hand side, equality must hold.
\end{proof}

\begin{proof}[End of the proof of Theorem \ref{mainresult}]
In order to conclude the proof, we are left to compute the extremal function of the trapezoid $K.$ Note that this is a convex real body but it is not centrally symmetric, hence the Baran formula (see eq. \eqref{Baran}) is not applicable for this case.

Instead we need to use another technique that has been suggested by Sione Ma`u.

We want to show that
\begin{align}\label{sioneclaim}
&V_K^*(\z)=\max\{V_{K_1}^*(\z_1),V_{K_2}^*(\z)\}\;\forall \z\in \C^2,\;\text{ where}\\
&K_1:=\{z\in \R: |z|\leq 1\},\;\;\;\;\;\;K_2:=\{\z\in \R^2: -R/r\leq z_1\leq 1,-rz_1-R\leq z_2\leq rz_1+R\}.
\end{align}
Let us denote the function appearing in the right hand side of \eqref{sioneclaim} by $v.$ Notice that $v$ is a good candidate for $V_K^*$, indeed it is fairly clear that $v\in\psh(\C^2)$ because it is the maximum of two plurisubharmonic functions on $\C^2$, in the same way $v\in \mathcal L^+(\C^2).$ We prove that the functions indeed coincide using the \emph{extremal ellipses technique}.

Let us pick $z\in \C^2\setminus K.$ Since $K$ is a real convex body, there exists at least one leaf $\mathcal E_z\ni z$ of the Monge Ampere foliation relative to $K.$ We recall \cite{BuLeMaRe10,BuLeMa09,BuLeMa15} that $\C^2$ is foliated by a set of complex analytic curves $\{\mathcal E_\alpha\}$ such that $V_K^*|_{\mathcal E_\alpha}$ is a subharmonic function that coincides with $V_{\mathcal E_\alpha\cap K}(\cdot,\mathcal E_\alpha).$ The curves $\mathcal E_\alpha$ are \emph{extremal ellipses}, i.e.,  (possibly degenerate) complex ellipses whose area is maximal among all ellipses of given direction and eccentricity that are inscribed in $K$.
We need the following lemma.
\begin{lemma}\label{SioneLemma}
If $\mathcal E$ is an extremal ellipse for $K$, then at least one of the following holds true
\begin{enumerate}[(i)]
\item $\mathcal E$ is an extremal ellipse for $K_2$ or
\item $\mathcal E\ni \z\mapsto V_{K_1}^*(\z_1)$ is harmonic on $\mathcal E\setminus \C\times K_1$ and the orientation of $\mathcal E$ is not in the direction $\z_2.$
\end{enumerate}
\end{lemma}
We postpone the proof of this claim and we show first why this implies that $V_K^*\equiv v.$

Assume \emph{(i)}. Then we have
\begin{equation}\label{onesideeq}
V_K^*(\z)|_{\mathcal E}=V^*_{\mathcal E\cap K}(\z,\mathcal E)=V_{K_2}^*(\z)|_{\mathcal E}\;\;\forall \z\in \mathcal E.
\end{equation}
Here the first equality is due to the fact that $\mathcal E$ is extremal for $K$ and the second  to the fact that it is extremal for $K_2.$

Assume now  \emph{(ii)}  We define the function $u:\mathcal E\rightarrow [-\infty,+\infty[$ by setting $u(\z):=V_{K_1}^*(\z_1).$ Notice that $u\in \mathcal L^+(\mathcal E),$ $u\leq 0$ on $\mathcal E\setminus (K_1\times \C)$ and $\ddc u=0$ on $\mathcal E\setminus (K_1\times \C)$ (i.e., the support of $\ddc u$ is in $\mathcal E\cap (K_1\times \C).$ Therefore, by the Global Domination Principle on the smooth algebraic variety $\mathcal E$, we have $u(\z)=V^*_{\mathcal E\cap (K_1\times \C)}(\cdot,\mathcal E).$ Then it follows that
\begin{equation}\label{othersideeq}
V_K^*(\z)|_{\mathcal E}=V^*_{\mathcal E\cap (K_1\times \C)}(\z,\mathcal E)=u(\z)=V_{K_1}^*(\z_1),\;\;\forall \z\in \mathcal E.
\end{equation}
For any (possibly degenerate) extremal ellipse $\mathcal E$ for $K$ we denote by $E$ the (possibly degenerate) \emph{filled-in} real ellipse relative to $\mathcal E.$ Note that $V_K^*(\cdot)|_{\mathcal E}=V_E^*|_{\mathcal E}=V_{\mathcal E\cap K}^*(\cdot,\mathcal E)$ since $\mathcal E$ is trivially an extremal ellipse for $E.$
 
If we assume \emph{(i)}, then the monotonicity of extremal functions with respect to the set inclusion implies
\begin{equation}\label{byinclusion1}
V_E^*(\z)\geq V_{K_1}^*(\z_1),\;\;\forall \z\in \C^2,
\end{equation}
while if we assume \emph{(ii)} the monotonicity property implies
\begin{equation}\label{byinclusion2}
V_E^*(\z)\geq V_{K_2}^*(\z),\;\;\forall \z\in \C^2.
\end{equation}
Finally, assumption \emph{(i)} implies \eqref{onesideeq} and \eqref{byinclusion1}, thus we have
\begin{equation*}
V_K^*(\z)=V_{K_2}^*(\z)\geq V_{K_1}(\z_1),\;\;\forall \z\in \mathcal E,
\end{equation*}
and assumption \emph{(ii)} implies \eqref{othersideeq} and \eqref{byinclusion2}, thus we have
\begin{equation*}
V_K^*(\z)=V_{K_1\times \C}^*(\z)\geq V_{K_2},\;\;\forall \z\in \mathcal E.
\end{equation*}
Therefore, on any extremal ellipse $V_K^*$ coincides with $v,$ but, since extremal ellipses for $K$ are a foliation of $\C^2,$ it follows that $V_K^*\equiv v$ everywhere.

Now we need to compute $V_{K_2}^*$.

Let us mention a possible way to compute based on \cite[Example 3.9]{Ba92}. Note that $V_{K_2}^*(z)=V_{K_2+(R/r,0)}^*(z+(R/r,0))$ and set $\tilde K_2:=K_2+(R/r,0).$ Then we can easily see that
\begin{align}
&\tilde K_2:=\{\z\in \R^2: 2\z\cdot y_k-1\in [-1,1],\forall k\in\{1,2,3\}\},\notag\\
&y_1:=\frac{1}{2(R+r)}\left(\begin{array}{c}r\\1 \end{array}\right),\;y_2:=\frac{1}{2(R+r)}\left(\begin{array}{c}r\\ -1 \end{array}\right),y_3:=y_1+y_2\notag.
\end{align}
\begin{equation}\label{bymodifiedbaran}
V_{\tilde K_2}^*(\z)=\log h\left[\max_{1\leq k\leq 3}\left(\sum_{l=1}^2 A_{k,l}|y_l\cdot \z|+|y_k\cdot \z-1|  \right) \right],\text{ where }A:=\left[\begin{array}{cc}1 &  0\\ 0 & 1\\ 1 & 1  \end{array}\right].
\end{equation}
Therefore we have
\begin{equation}
V_{K_2}^*(\z)=\log h\left[\max_{1\leq k\leq 3}\left(\sum_{l=1}^2 A_{k,l}|y_l\cdot (\z+\z^{(0)})|+|y_k\cdot (\z+\z^{(0)})-1|  \right) \right],
\end{equation}
where $\z^{(0)}:=(R/r,0).$
On the other hand, it is probably more easy to consider the linear map $\R^2\rightarrow \R^2$ represented by the invertible matrix 
$$L:=\left[\begin{array}{cc}\frac{r}{2(R+r)}& \frac{1}{2(R+r)}\\ \frac{r}{2(R+r)}&\frac{-1}{2(R+r)}\end{array}  \right]$$
and to notice that $L (K_2+\z^{(0)})=L K_2=\Sigma:=\{x\in \R^2:x_i\geq 0,\; x_1+x_2\leq 1\},$ the standard simplex. Hence we have
\begin{equation}\label{bysimplex}
V_{K_2}^*(\z)=V_\Sigma^*(L(\z+\z^{(0)}))=\log h(|s_1|+|s_2|+|s_1+s_2-1|)\Big|_{s=L(\z+\z^{(0)})}.
\end{equation}
Here the last equality is a classical result, see for instance \cite{BuLeMa15}. Note that in particular this shows that $V_K^*$ is continuous and $V_K^*((t_1,t_2))=V_K^*((t_1,-t_2)),$ as assumed in the proof of Proposition \ref{piEK}.

By Proposition \ref{piEK}, and using $V_{K_1}^*(\z_1)=\log h\left(\frac{|\z_1+1|}{2}+\frac{|\z_1-1|}{2}\right)$, we get
\begin{align*}
&V_{\pi E,\Sigma_{2,1}}(t)=2V_K(\Phi(t))\\
=&2\max\left[\log h\left(\frac{|\frac{t_1-R+r}{r}|}{2}+\frac{|\frac{t_1-R-r}{r}|}{2}\right)\;,V_\Sigma^*\left(\left(\frac{t_1+\sqrt{t_2}}{2(R+r)},\frac{t_1-\sqrt{t_2}}{2(R+r)}\right)\right)\right],\;\forall t\in \C^2.
\end{align*}
By equation \eqref{VPVS} we get, $\forall w\in \mathcal C$,
\begin{equation*}
V_E^*(w,\mathcal C)=2\max\left[\log h\left(\frac{|\frac{w_1-R+r}{r}|}{2}+\frac{|\frac{w_1-R-r}{r}|}{2}\right)\;,V_\Sigma^*\left(\left(\frac{w_1+\sqrt{w_3}}{2(R+r)},\frac{w_1-\sqrt{w_3}}{2(R+r)}\right)\right)\right].
\end{equation*}
By equation \eqref{VTVP} we obtain, $\forall z\in \mathcal T$, 
\begin{align*}
&V_{\mathbb T}^*(z,\mathcal T)\\
=&\max\Bigg[\log h\left(\frac{|\frac{\Psi_1(z)-R+r}{r}|}{2}+\frac{|\frac{\Psi_1(z)-R-r}{r}|}{2}\right)\;,\\
&\;\;\;\;\;\;\;\;\;\;\;\;\;\;\;\;V_\Sigma^*\left(\left(\frac{\Psi_1(z)+\sqrt{\Psi_3(z)}}{2(R+r)},\frac{\Psi_1(z)-\sqrt{\Psi_3(z)}}{2(R+r)}\right)\right)\Bigg]\\
=&\log h\Bigg\{\max\Bigg[ \frac{|\sqrt{1-(z_3/r)^2}+1|+|\sqrt{1-(z_3/r)^2}-1|}{2},\\
&\;\;\;\;\;\;\;\;\;\;\;\;\;\;\;\;\left|\frac{\sqrt{z_1^2+z_2^2}+z_2}{2(r+R)}\right|+\left|\frac{\sqrt{z_1^2+z_2^2}-z_2}{2(r+R)}\right|+\left|\frac{\sqrt{z_1^2+z_2^2}}{(r+R)}-1\right|\Bigg] \Bigg\}.
\end{align*}
Here we used the monotonicity of $h$ on the positive real semi-axis, equation \eqref{bysimplex}, the definition of $\Psi$, and the equation of the torus in the form \eqref{torusdef2}.
\end{proof}
\begin{proof}[Proof of Lemma \ref{SioneLemma}]
Let us first consider non-degenerate ellipses. It is easy to see that, if $\mathcal E$ is an extremal ellipse for $K$, then $\Card(\mathcal E\cap \partial K)\geq 2.$ Note that, if $\mathcal E$ intersects $\partial K$ only on the two diagonal edges or in the oblique edges and on the side lying on $\z_1=-1$  it can not be extremal for $K.$ For, we can simply translate the ellipse along the $z_1$ axis by a small but positive displacement (e.g., $\mathcal E':=\mathcal E+ \epsilon e_1$) and then slightly dilate the ellipse to get a larger ellipse $\mathcal E'':=(1+\delta)\mathcal E'$ that is still in $K$; this essentially follows by the strict convexity of the ellipse. It is even more clear that if $\mathcal E$ intersects $\partial K$ only on two adjacent edges can not have maximal area. As a consequence, for a non-degenerate maximal ellipse for $K$ there are only two possible configurations
\begin{enumerate}
\item $\Card(\mathcal E\cap \partial K)=4$, one point on each edge of $K$
\item $\Card(\mathcal E\cap \partial K)=3$, two points on oblique edges and one on $\z_1=1$
\item $\Card(\mathcal E\cap \partial K)=3$, two points on vertical edges and one on an oblique edge
\item $\Card(\mathcal E\cap \partial K)=2$, two points on vertical edges and major axis on $\z_2=0$.  
\end{enumerate} 
It is fairly clear that $(3)$ and $(4)$ implies (ii). Indeed, in each of this cases, $\z\mapsto V_{K_1}^*(\z_1)$ is pluriharmonic on $\C^2\setminus (K_1\times \C).$ 
On the other hand, in the cases $(1)$ and $(2)$ the ellipse $\mathcal E$ is tangent on at least three sides of $\partial K_2$, so it needs to be maximal for it.

Let us consider the case of degenerate maximal ellipses, that are line segments. If $\mathcal E$ is a maximal line for $K$, then $\mathcal E\cap \partial K$ can not be the set of one point on a oblique edge and one point on the vertical edge $z_1=-1.$ For, consider any small enough $\epsilon>0$ and let $\mathcal E'=\mathcal E+\epsilon n$, where $n$ is the unit normal to $\mathcal E$ pointing the half plane containing the barycenter of $K.$ Then we have $length(\mathcal E\cap K)<length(\mathcal E'\cap K).$ Also, if $\mathcal E$ is a vertical line , then it must be the side of $\partial K$ with $z_1=1.$ Therefore, for a maximal line for $K$ there are only the following cases
\begin{enumerate}
\item $\mathcal E\cap \partial K=\{z\in K:z_1=1\}$
\item $\mathcal E\cap \partial K=\{z\in\partial K,z_2=\pm r z_1\pm R\}$
\item $\Card(\mathcal E\cap \partial K)=2$, two points on the oblique sides
\item $\Card(\mathcal E\cap \partial K)=2$, two points on the vertical sides
\end{enumerate}
In the cases (1) and (3) $\mathcal E$ is evidently extremal for $K_2$, so (i) holds true. In the cases (2) and (4), since the function $\z\mapsto V_{K_1}^*(\z)$ is pluriharmonic on $\C^2\setminus (K_1\times \C)$ and $(\mathcal E\cap K)\subset (K_1\times \C)$ 
, (ii) holds.
\end{proof}
\bibliographystyle{abbrv}
\bibliography{references}

\begin{thebibliography}{10}

\bibitem{Ba92b}
M.~Baran.
\newblock Bernstein type theorems for compact sets in {${\bf R}^n$}.
\newblock {\em J. Approx. Theory}, 69(2):156--166, 1992.

\bibitem{Ba92}
M.~Baran.
\newblock Plurisubharmonic extremal functions and complex foliations for the
  complement of convex sets in {${\bf R}^n$}.
\newblock {\em Michigan Math. J.}, 39(3):395--404, 1992.

\bibitem{Ba16}
T.~Bayraktar.
\newblock Equidistribution of zeros of random holomorphic sections.
\newblock {\em Indiana Univ. Math. J.}, 65(5):1759--1793, 2016.

\bibitem{Ba17}
T.~Bayraktar.
\newblock Zero distribution of random sparse polynomials.
\newblock {\em Michigan Math. J.}, 66(2):389--419, 2017.

\bibitem{Be82}
E.~Bedford.
\newblock The operator {$(dd^{c})^{n}$} on complex spaces.
\newblock In {\em Seminar {P}ierre {L}elong-{H}enri {S}koda ({A}nalysis),
  1980/1981, and {C}olloquium at {W}imereux, {M}ay 1981}, volume 919 of {\em
  Lecture Notes in Math.}, pages 294--323. Springer, Berlin-New York, 1982.

\bibitem{BeTa82}
E.~Bedford and B.~A. Taylor.
\newblock A new capacity for plurisubharmonic functions.
\newblock {\em Acta Math.}, 149(1-2):1--40, 1982.

\bibitem{BlBoLeMaPi18}
T.~Bloom, L.~Bos, N.~Levenberg, S.~Ma`u, and F.~Piazzon.
\newblock The extremal function for the complex ball for generalized notions of
  degree and multivariate polynomial approximation.
\newblock {\em Annales Polonici Mathematici}, 121:to appear, 2018.

\bibitem{BlLe13}
T.~Bloom and N.~Levenberg.
\newblock Pluripotential energy and large deviation.
\newblock {\em Indiana Univ. Math. J.}, 62(2):523--550, 2013.

\bibitem{BlLeWi15}
T.~Bloom, N.~Levenberg, and F.~Wielonsky.
\newblock Logarithmic potential theory and large deviation.
\newblock {\em Comput. Methods Funct. Theory}, 15(4):555--594, 2015.

\bibitem{BoLe17p}
L.~Bos and N.~Levenberg.
\newblock Bernstein-{W}alsh theory associated to convex bodies and applications
  to multivariate approximation theory.
\newblock {\em Computational Methods and Function Theory}, Oct 2017.

\bibitem{BoLeMaPi17}
L.~Bos, N.~Levenberg, S.~Ma`u, and F.~Piazzon.
\newblock A weighted extremal function and equilibrium measure.
\newblock {\em Math. Scand.}, 121(2):243--262, 2017.

\bibitem{BoLeWa08}
L.~Bos, N.~Levenberg, and S.~Waldron.
\newblock Pseudometrics, distances and multivariate polynomial inequalities.
\newblock {\em J. Approx. Theory}, 153(1):80--96, 2008.

\bibitem{BuLeMa05}
D.~Burns, N.~Levenberg, and S.~Ma'u.
\newblock Pluripotential theory for convex bodies in {$\bold R^N$}.
\newblock {\em Math. Z.}, 250(1):91--111, 2005.

\bibitem{BuLeMa09}
D.~Burns, N.~Levenberg, and S.~Ma'u.
\newblock Exterior {M}onge-{A}mp\`ere solutions.
\newblock {\em Adv. Math.}, 222(2):331--358, 2009.

\bibitem{BuLeMaRe10}
D.~Burns, N.~Levenberg, S.~Ma'u, and S.~R\'{e}v\'{e}sz.
\newblock Monge-{A}mp\`ere measures for convex bodies and {B}ernstein-{M}arkov
  type inequalities.
\newblock {\em Trans. Amer. Math. Soc.}, 362(12):6325--6340, 2010.

\bibitem{BuLeMa15}
D.~M. Burns, N.~Levenberg, and S.~Ma`u.
\newblock Extremal functions for real convex bodies.
\newblock {\em Ark. Mat.}, 53(2):203--236, 2015.

\bibitem{Ce98}
U.~Cegrell.
\newblock Pluricomplex energy.
\newblock {\em Acta Math.}, 180(2):187--217, 1998.

\bibitem{Ch89}
E.~M. Chirka.
\newblock {\em Complex analytic sets}, volume~46 of {\em Mathematics and its
  Applications (Soviet Series)}.
\newblock Kluwer Academic Publishers Group, Dordrecht, 1989.
\newblock Translated from the Russian by R. A. M. Hoksbergen.

\bibitem{De85}
J.-P. Demailly.
\newblock Mesures de {M}onge-{A}mp\`ere et caract\'erisation g\'eom\'etrique
  des vari\'et\'es alg\'ebriques affines.
\newblock {\em M\'em. Soc. Math. France (N.S.)}, (19):124, 1985.

\bibitem{De07}
J.-P. Demailly.
\newblock {\em Complex analytic and differential geometry}.
\newblock online at:
  http://www-fourier.ujf-grenoble.fr/~demailly/manuscripts/agbook.pdf, 2007
  (with updates).

\bibitem{Kl91}
M.~Klimek.
\newblock {\em Pluripotential theory}, volume~6 of {\em London Mathematical
  Society Monographs. New Series}.
\newblock The Clarendon Press, Oxford University Press, New York, 1991.
\newblock Oxford Science Publications.

\bibitem{LeSz91}
L.~Lempert and R.~Sz\H{o}ke.
\newblock Global solutions of the homogeneous complex {M}onge-{A}mp\`ere
  equation and complex structures on the tangent bundle of {R}iemannian
  manifolds.
\newblock {\em Math. Ann.}, 290(4):689--712, 1991.

\bibitem{Le06}
N.~Levenberg.
\newblock Approximation in {$\Bbb C^N$}.
\newblock {\em Surv. Approx. Theory}, 2:92--140, 2006.

\bibitem{Le12}
N.~Levenberg.
\newblock Ten lectures on weighted pluripotential theory.
\newblock {\em Dolomites Research Notes on Approximation}, 5 (special
  issue):1--59, 2012.

\bibitem{LePe18}
N.~Levenberg and M.~Perera.
\newblock A global domination principle for p-pluripotential theory.
\newblock 2018.

\bibitem{PaWo91}
G.~Patrizio and P.-M. Wong.
\newblock Monge-{A}mp\`ere functions with large center.
\newblock In {\em Several complex variables and complex geometry, {P}art 2
  ({S}anta {C}ruz, {CA}, 1989)}, volume~52 of {\em Proc. Sympos. Pure Math.},
  pages 435--447. Amer. Math. Soc., Providence, RI, 1991.

\bibitem{Pi18}
F.~Piazzon.
\newblock Pluripotential numerics.
\newblock {\em Constructive Approximation}, Jun 2018.

\bibitem{Ru68}
W.~Rudin.
\newblock A geometric criterion for algebraic varieties.
\newblock {\em J. Math. Mech.}, 17:671--683, 1967/1968.

\bibitem{Sa82}
A.~Sadullaev.
\newblock Estimates of polynomials on analytic sets.
\newblock {\em Izv. Akad. Nauk SSSR Ser. Mat.}, 46(3):524--534, 671, 1982.

\bibitem{Si81}
J.~Siciak.
\newblock Extremal plurisubharmonic functions in {${\bf C}^{n}$}.
\newblock {\em Ann. Polon. Math.}, 39:175--211, 1981.

\bibitem{Tr17}
L.~N. Trefethen.
\newblock Multivariate polynomial approximation in the hypercube.
\newblock {\em Proc. Amer. Math. Soc.}, 145(11):4837--4844, 2017.

\bibitem{Ze91}
A.~Z\'eriahi.
\newblock Fonction de {G}reen pluricomplexe \`a p\^ole \`a l'infini sur un
  espace de {S}tein parabolique et applications.
\newblock {\em Math. Scand.}, 69(1):89--126, 1991.

\end{thebibliography}

\end{document}